\documentclass[preprint,12pt,authoryear]{elsarticle}

\usepackage{amssymb}

\usepackage{booktabs}
\usepackage{lineno}
\usepackage{amsmath,tabu}
\usepackage{amsthm}
\usepackage{latexsym, enumerate,amssymb}
\usepackage[usenames, dvipsnames]{color}
\usepackage{amstext} 
\usepackage{array}   
\usepackage{graphicx}
\usepackage{caption}
\usepackage{subcaption}
\usepackage{dsfont}
\newtheorem{lem}{Lemma}[section]
\newtheorem{thm}[lem]{Theorem}
\journal{Journal}

\begin{document}
	
	\begin{frontmatter}
		
		
		
		\title{Stochastic Finite Volume Approximation with Clustering in the Parameter Space for the Forward Uncertainty Quantification of Differential Equations with Random Parameters}
		
		
		\address[1]{Research Centre for Mathematics and Interdisciplinary Sciences, Shandong University, Qingdao, Shandong Province, 266237, China}
		\address[2]{School of Mathematics and Statistics, Changsha University of Science and Technology, Changsha, Hunan province, 410114, China }
		\address[3]{Frontiers Science Center for Nonlinear Expectations, Minister of Education, Shandong University, Qingdao, Shandong Province, 266237, China}
		\address[4]{Shandong Province Key Laboratory of Financial Risk, Qingdao, Shandong Province, 266237, China}
		\address[5]{Suzhou Research Institute of Shandong University, Suzhou, Jiangsu Province, 215123, China}

		\author[1,3,4,5]{Zhao Zhang\corref{cor1}}
		\ead{zhaozhang@sdu.edu.cn}
		\author[1,3]{Mengyao Xia\corref{cor1}}
		\ead{xmysd327@mail.sdu.edu.cn}
		\author[2]{Na Ou\corref{cor1}}
		\ead{oyoungla@csust.edu.cn}
		\cortext[cor1]{Corresponding author}

		\begin{abstract}
			The uncertainty quantification (UQ) for mathematical models with random parameters is important for many science and engineering problems. Forward UQ quantifies the impact of random parameters on the output of system. In the current study, we propose a new stochastic finite volume (SFV) scheme by combining SFV with clustering algorithm in the parameter space such that each cluster can be regarded as a finite volume with implicit boundaries. The advantage of SFV is that no specific form of the random variable is required such that discontinuous solutions and sharp interfaces can be accurately approximated. Compared to classic SFV based on structured grids, the new SFV-cluster scheme extends SFV to parameter spaces of higher dimensions. For demonstration and validation, we present the construction of SFV schemes for the Kraichnan-Orszag three-mode problem and the Buckley-Leverett equation. The error analysis of SFV is extended and the new algorithm is validated by numerical experiments.
		\end{abstract}
		
		
	\end{frontmatter}
	
	
	\section{Introduction}
	\label{intro}
	The uncertainty quantification (UQ) of parametric differential equations approximating physical processes with random parameters is important for the accurate and reliable predictions of many physical problems \citep{smith2024uncertainty}. UQ mainly consists of two aspects. The first aspect is forward UQ which concerns about the propagation of uncertainty from the input random parameters to the output solutions. The other aspect is inverse UQ which concerns the approximation of posterior uncertainty of random parameters given the observed data of model outputs.
	The two aspects of UQ are in fact closely connected through the Bayes' rule, since the likelihood function can be approximated using a surrogate built by forward UQ such that the posterior density of random parameters can be efficiently evaluated.
	
	There have been many studies in the area of forward UQ. The stochastic Garlerkin method \citep{lord2014, Newsum2017} is an intrusive UQ method where the numerical scheme for solving the governing equations of the physical process need to be modified to a large extent. The resulting linear equation system is very large even for moderate-dimensional problems coupling the physical and parameter spaces spanned by the basis functions. On the other hand, the stochastic collocation method \citep{Xiu2010, Yan2017,yang2021stochastic} is a non-intrusive UQ method where the forward model can be regarded as a black-box and polynomial chaos expansion is adopted for interpolating the parameter space. The two spectral methods reply on smooth polynomials and are hence difficult to approximate discontinuous solutions or sharp interfaces. To alleviate this restriction, there have been many studies such as the multi-element polynomial chaos, adaptive multiscale finite element, and filtering methods \citep{le2004multi, Wan2006, Witteveen2008, Foo2010, Zhong2022, Vauchel2023, Xiao2023}.
	
	The stochastic finite volume (SFV) method proposed firstly in \citet{abgrall2007simple} approximate solutions by building grids in the paramter space. SFV only requires moderate modification of the numerical scheme for deterministic problems, and is also referred to a semi-intrusive UQ method \citep{Jin2017}. The advantage of SFV is that no specific form of the random variable is required such that discontinuous solutions and sharp interfaces can be accurately approximated. However, for even moderate-dimensional random paramters, the computational cost of structured-grid-based SFV becomes prohibitively high as the number of grid cells as finite volumes grows exponentially with the number of dimensions of the parameter space \citep{Abgrall2014,Geraci2016}. Building unstructured grids seem to be a remedy for the problem, but there is yet no algorithm for building unstructured grids in parameter spaces with dimension larger than three.
	
	In the current study, we propose a new SFV-cluster scheme where grid cells in the parameter space are built by clustering samples, such that the clusters can be regarded as unstructured grid cells with implicit boundaries. The conditional expectation in each cell treated as a finite volume is approximated using the samples in the cell. The approximation error of the new scheme is discussed and the new algorithm is validated on test cases of the Kraichnan-Orszag three-mode problem and the Buckley-Leverett example.
	
	The paper is organised as follows. Section 2 presents the development of SFV method. Section 3 presents the clustering algorithm in the parameter space. Section 4 shows the error analysis. Section 5 presents the test case. Section 6 is the conclusion of the current study. 
	
	\section{The Stochastic Finite Volume Scheme}
	\subsection{Basic Idea of SFV}
	Consider a general parametric governing equation
	\begin{equation}
		\mathcal{L}(u, \textbf{y})=0~,
		\label{general}
	\end{equation}
	where $\mathcal{L}$ is a operator containing initial and boundary conditions, $u=u(\textbf{x},t,\textbf{y})$ is the dependent variable, $\textbf{x}$ and $t$ are spatial and temporal coordinates, $\textbf{y}$ is a random parameter that might be dependent on $\textbf{x}$ and $t$ in some problems. For deterministic equations, the idea of finite volume (FV) method is to integrate the governing equation on each cell in the physical space. For parametric problems, the basic idea of SFV is to discretize the physical space as well as the parameter space. A set $\{\omega_j\}, j=1,...,N$ is built such that for probability measure $P$, both $P(\omega_i \cap \omega_j)=0$ $\forall i\neq j$ and $\Omega=\cup_{j=1}^{N}\omega_j$ hold. Then, conditional expectation $E(u_h|\omega_j)$ is used to approximate the solution of parametric equations \citep{abgrall2007simple}, where $u_h$ is the numerical solution in the physical space, and $E(u_h|\omega_j)$ is defined as 
	\begin{equation}
		E(u_h|\omega_j)=\frac{\int_{\omega_j}u_h dP}{\int_{\omega_j}dP}~.
	\end{equation}
	
	To compute $E(u_h|\omega_j)$, the governing parametric equation needs to be integrated on the cells both in the physical and paramter spaces. \citet{Jin2017} presents the derivation of SFV scheme for a basic hyperbolic conservation law. In this section, we present the construction of SFV schemes for two benchmark models used in the current study, the Kraichnan-Orszag three-mode problem and Buckley-Leverett example.
	
	\subsection{SFV scheme for Kraichnan-Orszag three-mode problem}
	The Kraichnan-Orszag three-mode problem is defined by the system of nonlinear ordinary differential equations
	\begin{equation}
		\frac{dz_1}{dt} = z_1 z_3 \quad \frac{dz_2}{dt} = -z_2 z_3 \quad \frac{dz_3}{dt} = -z_1^2 + z_2^2
		\label{K-O}
	\end{equation}
	with stochastic initial conditions
	\begin{align}
		z_1(0) = z_1(0;y) \quad z_2(0) = z_2(0;y) \quad z_3(0) = z_3(0;y)
	\end{align}
	where $y$ represents the random parameter , and $z_i = z_i(t, y)$ $(i = 1, 2, 3)$ are the stochastic solution components.
	Integrating Eq.~\eqref{K-O} over cells treated as finite volumes in both time and parameter spaces yields  
	\begin{equation}
		\int_{\omega_k} \int_{\Delta t^n} \frac{dz_i}{dt} \rho(y) \, dt \, dy = \int_{\omega_k} \int_{\Delta t^n} f_i(z_1, z_2, z_3) \rho(y) \, dt \, dy~, 
	\end{equation}
	for $i=1, 2, 3$, where $\Delta t^n$ and $\omega_j$ are the cells in the time and parameter spaces, respectively, $\rho(y)$ is the probability density function of $y$, and $f_i$ represents the right-hand side of Eq.~\eqref{K-O} such that
	\begin{equation}
		f_1 = z_1 z_3, \quad f_2 = -z_2 z_3, \quad f_3 = -z_1^2 + z_2^2~. 
	\end{equation}
	
	Adopting the forward Euler scheme in time leads to
	\begin{equation}
		\int_{\omega_k} \bigl(z_i(t_{n+1},y) - z_i(t_n,y)\bigr) \rho(y)\,dy
		= \Delta t \int_{\omega_k} f_i\bigl(z_1(t_{n},y),z_2(t_{n},y),z_3(t_{n},y)\bigr) \rho(y)\,dy~.
	\end{equation}
	A quadrature rule can be adopted to approximate integration on $\omega_k$ as
	\begin{eqnarray}
		& \sum_{m} \left[ {z_i}(t_{n+1}, y_{m}) - {z_i}(t_n, y_{m}) \right] \rho( y_{m})\beta_{m} \\
		= &\Delta t_n \sum_{m} f_i(z_1(t_n, y_{m}),z_2(t_n, y_{m}),z_3(t_n, y_{m}))\rho( y_{m}) \beta_{m}
		\label{ex1eq}
	\end{eqnarray}
	where $\beta_{m}$ denote the weights in $\omega_k$. Further, assume the solution is piecewise constant on the cells in the parameter space such that the governing equations can be approximated as
	\begin{equation}
		(\bar{z}_{i,k}^{n+1}-\bar{z}_{i,k}^{n})P(\omega_k) = \Delta t_n\sum_{m=1}^M f_i(z_1(t_n, y_{m}),z_2(t_n, y_{m}),z_3(t_n, y_{m})) \rho(y_{m}) \beta_{m}~,\label{ge1}
	\end{equation}
	where $\bar{z}_{i,j}^{n}$ is the mean ${z_i}(t_{n}, y)$ on the $j'$th cell $\omega_j$, and $P(\omega_k)=\sum_{m} \rho( y_{m})\beta_{m}$ is the approximated probability measure of $\omega_k$.
	
	
	\subsection{SFV scheme for the hyperbolic Buckley-Leverett equation}
	The governing equation for the hyperbolic Buckley-Leverett benchmark with random parameter $y$ is
	\begin{equation}
		\frac{\partial u}{\partial t} + \frac{\partial f(u;y)}{\partial x} = 0, \quad x \in (0, L)
		\label{buckley_leverett}
	\end{equation}
	where $u=u(x,t;y)$ and 
	\begin{equation}
		f(u;y) = \frac{u^2}{u^2 + \alpha(y)(1-u)^2}~.
		\label{flux_function}
	\end{equation}
	The initial condition is
	\begin{equation}
		u(x, 0) = u_0(x) =
		\begin{cases} 
			u_L, & \text{if } x < x_0, \\
			u_R, & \text{if } x > x_0.
		\end{cases}
		\label{Riemann initial data}
	\end{equation}
	Integrating Eq.~\eqref{buckley_leverett} over cells in both physical and parameter spaces yields
	\begin{equation}
		\int_{\Delta x} \int_{\omega_k} \frac{\partial u}{\partial t} \rho(y) \, dy \, dx + \int_{\Delta x} \int_{\omega_k} \frac{\partial f(u;y)}{\partial x} \rho(y) \, dy \, dx = 0.
		\label{buckley_leverett_integrate}
	\end{equation}
	where \(\Delta x_i = [x_{i-1/2}, x_{i+1/2}] \) is a physical space cell.  Let $\bar{u}_{i,k}^n$ denote the mean $u$ in the $i$'th physical cell and $k$'th parameter cell at the $n$'th time step, adopting the forward Euler scheme in time, the SFV scheme is
	\begin{equation}
		|\Delta x_i| P(\omega_k) \frac{\bar{u}_{i,k}^{n+1} - \bar{u}_{i,k}^n}{\Delta t}+ \int_{\omega_k} \left[ f(u(x_{i+1/2},t_n,y);y) - f(u(x_{i-1/2},t_n,y);y) \right] \rho(y) \, dy = 0.
		\label{eq:cell-average-evolution}
	\end{equation}
	where $|\Delta x|$ is the volume of $\Delta x$, $P(\omega_k) = \int_{\omega_k} \rho(y) \, dy$ is the probability measure of $\omega_k$. Adopting the fourth-order Runge–Kutta method (RK4) , the SFV scheme becomes
	\begin{equation}
		\begin{aligned}
			\bar{u}_{i,k}^{\,n+1}
			&= \bar{u}_{i,k}^{\,n}
			+ \frac{\Delta t}{6}\Bigg[
			\mathcal{L}_{i,k}\!\left(\bar{u}^{\,n}\right)
			+ 2\mathcal{L}_{i,k}\!\left(\bar{u}^{\,n}+\frac{\Delta t}{2}\mathcal{L}_{i,k}(\bar{u}^{\,n})\right) \\
			&\quad + 2\mathcal{L}_{i,k}\!\left(
			\bar{u}^{\,n}+\frac{\Delta t}{2}\mathcal{L}_{i,k}\!\left(
			\bar{u}^{\,n}+\frac{\Delta t}{2}\mathcal{L}_{i,k}(\bar{u}^{\,n})
			\right)
			\right) \\
			&\quad + \mathcal{L}_{i,k}\!\left(
			\bar{u}^{\,n}+\Delta t\,\mathcal{L}_{i,k}\!\left(
			\bar{u}^{\,n}+\frac{\Delta t}{2}\mathcal{L}_{i,k}\!\left(
			\bar{u}^{\,n}+\frac{\Delta t}{2}\mathcal{L}_{i,k}(\bar{u}^{\,n})
			\right)
			\right)
			\right)
			\Bigg],
		\end{aligned}
		\label{eq:sfv-rk4}
	\end{equation}
	where  \(\mathcal{L}_{i,k}\) is defined as
	\[
	\mathcal{L}_{i,k}(u)
	= -\frac{1}{|\Delta x_i|\,P(\omega_k)}
	\int_{\omega_k}
	\left[ f\big(u(x_{i+1/2},t,y);y)
	- f\big(u(x_{i-1/2},t,y);y) \right]
	\rho(y)\,\mathrm{d}y.
	\]
	The numerical flux $f(u(x_{i+1/2},t_n,y);y)$ at the interface of two physical cells can be approximated by the upwind scheme. For SFV, a quadrature rule is adopted to approximate the integration over $\omega_k$ in Eq.~\eqref{eq:cell-average-evolution}. Using quadrature nodes $y_m$ and weights $\beta_m$ such that $\sum_m \rho( y_{m})\beta_m = P(\omega_k)$, Eq.~\eqref{eq:cell-average-evolution} can be approximated as
	\begin{equation}
		|\Delta x_i| P(\omega_k) \frac{\bar{u}_{i,k}^{n+1} - \bar{u}_{i,k}^n}{\Delta t}+ \sum_{m} \left[ f(u(x_{i+1/2},t_n,y_m);y_m) - f(u(x_{i-1/2},t_n,y_m);y_m) \right] \rho( y_{m})\beta_m = 0.
	\end{equation}
	Let
	\begin{equation}
		F_{i+1/2,k} = \sum_{m} f(u(x_{i+1/2},t_n,y_m);y_m)\rho( y_{m})\beta_m, \quad F_{i-1/2,k} = \sum_{m} f(u(x_{i-1/2},t_n,y_m);y_m)\rho( y_{m})\beta_m~,
		\label{eq:flux-quadrature}
	\end{equation}
	we have the discrete scheme adopting forward Euler as
	\begin{equation}
		|\Delta x_i| P(\omega_k) \frac{u_{i,k}^{n+1} - u_{i,k}^n}{\Delta t} + F_{i+1/2,k}^n - F_{i-1/2,k}^n = 0.
		\label{eq:continous}
	\end{equation}
	The discrete scheme using the RK4 method is given by
	\begin{equation}
		\bar{u}_{i,k}^{\,n+1}
		= \bar{u}_{i,k}^{\,n}
		+ \frac{\Delta t}{6}\left(K_{i,k}^{(1)}+2K_{i,k}^{(2)}
		+2K_{i,k}^{(3)}+K_{i,k}^{(4)}\right),
		\label{eq:rk4-update}
	\end{equation}
	where the stage residuals are defined as
	\begin{alignat}{2}
		K_{i,k}^{(1)} &= \mathcal{L}_{i,k}\!\left(\bar{u}^{\,n}\right), &\qquad
		K_{i,k}^{(2)} &= \mathcal{L}_{i,k}\!\left(\bar{u}^{\,n}+\tfrac{\Delta t}{2}K^{(1)}\right), \notag\\
		K_{i,k}^{(3)} &= \mathcal{L}_{i,k}\!\left(\bar{u}^{\,n}+\tfrac{\Delta t}{2}K^{(2)}\right), &
		K_{i,k}^{(4)} &= \mathcal{L}_{i,k}\!\left(\bar{u}^{\,n}+\Delta t\,K^{(3)}\right),
		\label{eq:rk4-stages}
	\end{alignat}
	where \(\mathcal{L}_{i,k}\) is defined as
	\begin{equation*}
		\mathcal{L}_{i,k}(u) = -\frac{1}{|\Delta x_i|\,P(\omega_k)}
		\left( F_{i+1/2,k} - F_{i-1/2,k} \right).
	\end{equation*}
	
	\section{SFV with Clustering in the Parameter Space}
	The SFV method is based on the discretisation of physical and parameter spaces. The existing SFV schemes employ structured grids in the parameter space, while the adaptive SFV approach allows structured grids with anisotropic resolution. For structured grids, the number of cells in the parameter space grows exponentially with the space dimension. This restricts the application of SFV to problems involving parameter spaces of dimensions higher than three. An obvious remedy for this problem is to build unstructured meshes in the parameter space. However, there is no algorithm yet for generating unstructured meshes in spaces with dimension higher than three.
	
	An important feature of the SFV schemes in Eq.~\eqref{ge1} and \eqref{eq:continous} is that the cells in the parameter space are decoupled, indicating that there is no flux across the boundary faces between them. The quadrature rule is based on the nodes inside each cell, while nodes in the parameter space are actually samples. Therefore, we can approximate the parameter space by a set of samples, and then cluster the samples into different clusters such that each cluster can be regarded as a cell with implicit boundaries. This allows us to generate unstructured meshes in spaces of dimension higher than three conveniently. 
	Hereafter, the new scheme combining SFV with clustering in the parameter space is referred to as SFV-cluster, while the vanilla SFV scheme based on structured grid is referred to as SFV.
	
	Suppose a number of $N$ samples are generated by the Monte Carlo method, and cell $\omega_j$ contains $M$ samples denoted by $y_m$, then we have $\rho(y_m)=1/N$ and $P(\omega_k)=M/N$. Further, let $w_m=1$ for simplicity, The SFV discrete Eq.~\eqref{ex1eq} for Kraichnan-Orszag three-mode problem becomes
	\begin{equation}
		(\bar{z}_{i,k}^{n+1}-\bar{z}_{i,k}^{n})M = \Delta t_n\sum_{m=1}^M f_i(z_1(t_n, y_{m}),z_2(t_n, y_{m}),z_3(t_n, y_{m}))~.\label{ge1}
	\end{equation}
	The flux approximation in Eq.~\eqref{eq:flux-quadrature} for the hyperbolic Buckley-Leverett problem becomes
	\begin{equation}
		F_{i+1/2,j} = \frac{1}{N}\sum_{y_m\in\omega_j} f(u(x_{i+1/2},t_n,y_m);y_m), \quad F_{i-1/2,j} = \frac{1}{N}\sum_{y_m\in\omega_j} f(u(x_{i-1/2},t_n,y_m);y_m)~.
		\label{eq:flux-quadrature2}
	\end{equation}
	
	The computational cost of SFV is proportional to the number of cells , rather than the number of samples in the parameter space. Clustering is a typical unsupervised machine learning method that aims to group similar samples together based on their intrinsic characteristics without 
	labels. There have been different methods for clustering \citep{xu2005survey,xu2015comprehensive}. Partitioning methods, such as K‑means and K‑medoids, divide samples into a predetermined number of clusters by optimization. Hierarchical methods create a tree‑like structure of clusters, either agglomeratively or divisively. Model‑based methods, such as the Gaussian Mixture Model, assume that the samples are generated from a mixture of probability distributions and provide soft cluster assignments. We only mention some clustering methods here, and there are other methods in literature.
	
	The K‑means method is one of the most widely used clustering algorithms.  
	The basic idea of K‑means is to partition n samples into $K$ clusters, where each sampple belongs to the cluster whose centroid  is closest. The algorithm seeks to minimize the objective as  
	\[
	\sum_{k=1}^{K} \sum_{y_m \in C_k} \|y_m - \mu_k\|^2,
	\]
	where \(C_k\) denotes the \(k\)-th cluster, $y_m$ is a sample and \(\mu_i\) is the centroid of \(C_k\). As minimizing this objective exactly is NP‑hard, K‑means employs an iterative approach that converges to a local minimum.
	The standard K‑means algorithm proceeds as follows.
	\begin{enumerate}
		\item Initialization: Choose \(K\) initial centroids, which can be selected randomly from the dataset. 
		\item Assignment: For each sample, compute its Euclidean distance to every centroid and assign the sample to the cluster whose centroid is the closest. This forms \(K\) clusters for the current iteration.
		\item Update: For each cluster, recalculate its centroid as the mean of all samples currently assigned to the cluster. 
		\item Iteration: Repeat the assignment and update steps until a termination condition is met—typically when the centroid positions change negligibly or after a fixed number of iterations.
	\end{enumerate}
	
	Since the K-means clustering result is affected by the initial centroids, it is common practice to run the algorithm multiple times with different initializations and then keep the result associated with the lowest objective.
	
	\section{Error analysis}
	The analysis here for SFV-Cluster is based on the error analysis in \citet{Jin2017} for SFV employing structured grid. 
	Assume \( p=p(x, t; y) \) is the exact solution and \( p_{i, k}^{n} \) is the approximated solution at $t=t_{n}$ using SFV. We denote it as \( p_{h}=\left\{p_{i k}^{n}\right\} \). Let the mean and variance of the exact solution at the point $(x_i, t_n)$ be
	\[
	\begin{aligned}
		\mathbb{E}[p]\left(x_{i}, t^{n}\right)&=\int_{\Omega} p\left(x_{i}, t^{n} ; y\right) \rho(y) dy,   \\
		\mathbb{V}[p]\left(x_{i}, t^{n}\right)&=\mathbb{E}\left[p^{2}\left(x_{i}, t^{n}\right)\right]-\left(\mathbb{E}[p]\left(x_{i}, t^{n}\right)\right)^{2},
	\end{aligned}
	\]
	respectively. Denote 
	\[
	\begin{aligned}
		\mathbb{E}[p_h]_{i}^{n}&=\int_{\Omega} \sum_{k}^{N_y}p_{i k}^{n} \rho(y) dy,   \\
		\mathbb{V}[p_h]_{i}^{n}&=\mathbb{E}\left[p_{h}^{2}\right]_{i}^{n}-\left(\mathbb{E}\left[p_{h}\right]_{i}^{n}\right)^{2},
	\end{aligned}
	\]
	the corresponding SFV-Cluster approximation is computed as
	\[
	\begin{aligned}
		E_{h}\left[p_{h}\right]_{i}^{n}&=\sum_{k=1}^{N_{y}} p_{i k}^{n} {\tilde \alpha_k},\\ 
		V_{h}\left[p_{h}\right]_{i}^{n}&=E_{h}\left[p_{h}^{2}\right]_{i}^{n}-\left(E_{h}\left[p_{h}\right]_{i}^{n}\right)^{2},
	\end{aligned}
	\]
	where ${\tilde\alpha_k}$ is the numerical approximation of $\alpha_{k}$ and 
	\[
	\begin{aligned}
		\alpha_{k}&=\int_{\omega_k} \rho(y) dy,  \\
		{\tilde\alpha_k}&=\frac{1}{N}\sum_{n=1}^N {\mathds 1}(y_n \in \omega_k)~.
	\end{aligned}.
	\]
	According to the Monte Carlo approximation, the error between ${\tilde\alpha_k}$ and $\alpha_{k}$ is
	\[
	\left|\alpha_{k}-{\tilde\alpha_k}\right| \lesssim O(\frac{\sigma_k}{\sqrt{N}}),
	\]
	where $\sigma_k$ is the standard variance of the random variable ${\mathds{1}(y\in \omega_k)}$. We use the moment estimator to estimate the value of $\sigma_k$. Denote the random variable $Z={\mathds{1}(y\in \omega_k)}$, the zero-order moment estimate is
	\[
	\bar z=\frac{m_k}{N},
	\]
	where $m_k$ is the number of samples belonging to the $k-$th cluster. We can then have the second-order moment estimate
	\[
	\begin{aligned}
		s_k^2&=\frac{1}{N-1}\sum_{i=1}^N (z_i-\bar z)^2\\
		&=\frac{1}{N-1}\left(\sum_{i=1}^{m_k} \left(1-\frac{m_k}{N}\right)^2+ \sum_{i=1}^{N-m_k}\left(0-\frac{m_k}{N}\right)^2 \right)\\
		&=\frac{(N-m_k)m_k}{N(N-1)},
	\end{aligned}
	\]
	which is an unbiased estimate of $\sigma_k^2$. When $m_k <N/2$, the less points contained in the $k-$th cluster, the smaller the estimated $\sigma_k$ is.

	\subsection{Estimates in $L_{\infty}$ -norm}
	Assume
	\[
	\left\|p-p_{h}\right\|_{L^{\infty}(D \times \Omega)} \leqslant C_{1} \Delta x^{q}+C_{2} \Delta y^{r},
	\]
	where $\Delta x$ is the mesh size in the physical domain, $\Delta y$ can be viewed as the largest radius of the clusters, which is used to measure the discretization error in the parameter space.
	
	\begin{thm}
		The error of the expectation and variance are
		\[
		\begin{aligned}
			\left\|\mathbb{E}[p]-E_{h}\left[p_{h}\right]\right\|_{L^{\infty}(D)} &\leqslant C_{1} \Delta x^{q}+C_{2} \Delta y^{r}+C_3\frac{\sigma N_y}{\sqrt{N}},\\
			\left\|\mathbb{V}[p]-V_{h}\left[p_{h}\right]\right\|_{L^{\infty}(D)} &\leqslant C_1^{'} \Delta x^{q}+C_2^{'} \Delta y^{r}+C_3^{'}\frac{\sigma N_y}{\sqrt{N}}.
			\label{thm1}
		\end{aligned}
		\]
		where $\sigma=\max{(\sigma_1,\cdots, \sigma_k,\cdots)}$ and $N_y$ is the number of clusters.
	\end{thm}
	\begin{proof}
		\[
		\begin{aligned}
			\left|\mathbb{E}[p](x_i, t^n)-E_{h}\left[p_{h}\right]_i^n\right|&=\left| \int_{Y} p(x_i, t^n;y)\rho(y)dy- \sum_k p_{i k}^{n} {\tilde \alpha_k} \right|\\
			&\leqslant \left| \int_{Y} p(x_i, t^n;y)\rho(y)dy- \sum_k p_{i k}^{n} {\alpha_k} \right|+\left|\sum_k p_{i k}^{n} ({\alpha_k}-{\tilde \alpha_k}) \right|\\
			&\lesssim \sup_{\omega_k}\left| p(x_i, t^n;y)-p_{ik}^n \right|+\frac{\sigma N_y}{\sqrt{N}}\sup_{\omega_k}\left|p_{ik}^n\right|,
		\end{aligned}
		\]
		and 
		\[
		\begin{aligned}
			\left\|\mathbb{E}[p]-E_{h}\left[p_{h}\right]\right\|_{L^{\infty}(D)} &\lesssim\left\|p-p_{h}\right\|_{L^{\infty}(D \times \Omega)}+\frac{\sigma N_y}{\sqrt{N}}\left\|p_{h}\right\|_{L^{\infty}(D \times \Omega)} \\
			&\leqslant C_{1} \Delta x^{q}+C_{2} \Delta y^{r}+C_3\frac{\sigma N_y}{\sqrt{N}}.
		\end{aligned}
		\]
		For the variance, we have
		\[
		\left\|\mathbb{V}[p]-V_{h}\left[p_{h}\right]\right\|_{L^{\infty}(D)} \leqslant \left\|\mathbb{E}\left[p^{2}\right]-E_{h}\left[p_{h}^{2}\right]\right\|_{L^{\infty}(D)}+\left\|(\mathbb{E}[p])^{2}-\left(E_{h}\left[p_{h}\right]\right)^{2}\right\|_{L^{\infty}(D)}
		\]
		and the first term can be estimated as
		\[
		\begin{aligned}
			\mathrm{E}\left[p^{2}\left(x_{i}, t^{n}\right)\right]-E_{h}\left[p_{h}^{2}\right]_{i}^{n}|&=\left| \int_{\Omega} p^{2}\left(x_{i}, t^{n} ; y\right) \rho(y) d y-\sum_{k}\left(p_{i k}^{n}\right)^{2} {\tilde\alpha_k} \right| \\
			&\leqslant\left|\sum_{k} \int_{\omega_{k}} p^{2}\left(x_{i}, t^{n} ; y\right) \rho(y) d y-\sum_{k}\left(p_{i k}^{n}\right)^{2} \int_{\omega_{k}} \rho(y) d y\right|\\
			&+\left|\sum_{k}\left(p_{i k}^{n}\right)^{2} \left(\int_{\omega_{k}} \rho(y) d y -{\tilde\alpha_k}\right) \right| \\
			&\lesssim C \sup _{\omega_{k}}\left|p\left(x_{i}, t^{n} ; y\right)-p_{i k}^{n}\right|+\sup_{\omega_k} \left|p_{ik}^n\right|^2\frac{\sigma}{\sqrt{N}} N_y~.
		\end{aligned}
		\]
		Hence,
		\[
		\left\|\mathbb{E}\left[p^{2}\right]-E_{h}\left[p_{h}^{2}\right]\right\|_{L^{\infty_{(D)}}} \lesssim C\left\|p-p_{h}\right\|_{L^{\infty_{(D \times \Omega)}}}+\left\|p_h^2\right\|_{L^{\infty}(D \times \Omega)}\frac{\sigma N_y}{\sqrt{N}}~.
		\]
		For the second term, we have
		\[
		\begin{aligned}
			\left\|(\mathbb{E}[p])^{2}-\left(E_{h}\left[p_{h}\right]\right)^{2}\right\|_{L^{\infty}(D)}
			&=\left\|(\mathbb{E}[p_h]+E_{h}\left[p_{h}\right])
			(\mathbb{E}[p_h]-E_{h}\left[p_{h}\right])\right\|_{L^{\infty}(D)}\\
			&\leqslant C \left\| \mathbb{E}[p]-E_{h}\left[p_{h}\right] \right\|_{L^{\infty}(D)}~.    
		\end{aligned}
		\]
		Therefore,
		\[
		\left\|\mathbb{V}[p]-V_{h}\left[p_{h}\right]\right\|_{L^{\infty}(D)} \leqslant C_1^{'} \Delta x^{q}+C_2^{'} \Delta y^{r}+C_3^{'}\frac{\sigma N_y}{\sqrt{N}}.
		\]
	\end{proof}
	
	It can be seen that if the number of clusters $N_y$ is small, $m_k$ would be large, i.e.,
	\[
	m_k \geq \frac{N}{N_y},
	\]
	The term $\sigma N_y$ on the right-hand side. Theorem~\ref{thm1} refines the error decomposition proposed by \citet{Jin2017} by incorporating the error introduced by the discrete quadrature samples drawn from a specific distribution in the parameter space. Theorem~\ref{thm1} shows that the maximum error over the parameter space is governed by the physical domain discretization, the parameter space discretization, and the quadrature error. For both the expectation error and the variance error, the orders of convergence rate from these three terms are the same. 

	\subsection{Estimates in $L_{1}$ -norm}
	Let \( p_h^y \) be the numerical solution which is exact in $x$ variable and discretized in $y$ and by $p_h^{xy}$ the numerical discretized in both variables. Assume that 
	\[
	\begin{array}{c}
		\left\|p_{h}^{y}-p_{h}^{x y}\right\|_{L^{1}(D)} \leqslant C_{1} \Delta x^{q}, \ \ \forall y \in \Omega, \\
		\left\|p-p_{h}^{y}\right\|_{L^{1}(\Omega)} \leqslant C_{2} \Delta y^{r}, \ \ \forall x \in D,
	\end{array}
	\]
	where $\Delta x$ is the mesh size in the physical domain, $\Delta y$ can be viewed as the largest radius of the clusters, which is used to measure the discretization error in the parameter space. Then the following estimates are established.
	\begin{thm}
		The error of the expectation and variance are
		\[
		\begin{aligned}
			\left\|\mathbb{E}[p]-E_h\left[p_{h}^{xy}\right]\right\|_{L^{1}(D)}
			&\leqslant C_1 \Delta x^{q}+C \Delta y^{r}+C_3\frac{\sigma}{\sqrt{N}},\\
			\left\|\mathbb{V}[p]-V_{h}\left[p_{h}\right]\right\|_{L^{1}(D)} &\leqslant C_1^{'} \Delta x^{q}+C^{'} \Delta y^{r}+C_3^{'}\frac{\sigma}{\sqrt{N}}.
		\end{aligned}
		\]
	\end{thm}
	\begin{proof}
		\[
		\begin{aligned}
			\left\|\mathbb{E}[p]-E_h\left[p_{h}^{xy}\right]\right\|_{L^{1}(D)}&=\left\|\mathbb{E}[p]-\mathbb{E}\left[p_{h}^{ y}\right]+\mathbb{E}\left[p_{h}^{ y}\right]-\mathbb{E}\left[p_{h}^{x y}\right]+\mathbb{E}\left[p_{h}^{x y}\right]-E_h\left[p_{h}^{x y}\right]\right\|_{L^{1}(D)} \\
			&\leqslant\left\|\mathbb{E}[p]-\mathbb{E}\left[p_{h}^{y}\right]\right\|_{L^{1}(D)}+\left\|\mathbb{E}\left[p_{h}^{y}\right]-\mathbb{E}\left[p_{h}^{xy}\right]\right\|_{L^{1}(D)}\\
			&\ \ \ +\left\|\mathbb{E}\left[p_{h}^{xy}\right]-E_h\left[p_{h}^{x y}\right]\right\|_{L^{1}(D)}~.
		\end{aligned}
		\]
		For the first integral, we have
		\[
		\begin{aligned}
			\left\|\mathbb{E}[p]-\mathbb{E}\left[p_{h}^{y}\right]\right\|_{L^{1}(D)}&=\int_{D}\left|\int_{\Omega}\left(p-p_{h}^{y}\right) \rho(y) dy\right| d x
			\leqslant \int_{D} \sup_{\Omega} \rho(y) \int_{\Omega}\left|p-p_{h}^{y}\right| d y d x\\
			&=\sup_{\Omega} \rho(y)\int_{D} \left\|p-p_{h}^{y}\right\|_{L^{1}(\Omega)} dx
			\leqslant \sup_{\Omega} \rho(y)|D|C_2 \Delta y^{r}\leqslant C\Delta y^{r},
		\end{aligned}
		\]
		where $|D|$ is the area of the physical domain.
		Similarly, for the second integral, we have
		\[
		\begin{aligned}
			\left\|\mathbb{E}\left[p_{h}^{y}\right]-\mathbb{E}\left[p_{h}^{xy}\right]\right\|_{L^{1}(D)}&=\int_{D}\left|\int_{\Omega}\left(p_{h}^{y}-p_{h}^{x y}\right) \rho(y) d y\right| dx\\
			&=\int_{\Omega}\left[\int_{D}\left|p_{h}^{y}-p_{h}^{x y}\right| d x\right] \rho(y) d y\\
			&=\int_{\Omega} \left\|p_{h}^{y}-p_{h}^{x y}\right\|_{L^{1}(D)} \rho(y) d y\\
			&\leqslant C_{1} \Delta x^{q}\int_{\Omega} \rho(y) d y\leqslant C_{1} \Delta x^{q}~.
		\end{aligned}
		\]
		The third term can be estimated as
		\[
		\begin{aligned}
			\left|\mathbb{E}[p_h]_i^n-E_{h}\left[p_{h}\right]_i^n\right|&=\left|\sum_k p_{i k}^{n} ({\alpha_k}-{\tilde \alpha_k}) \right|\\
			&\leqslant \frac{\sigma}{\sqrt{N}}\sum_k \left|p_{ik}^n\right|,
		\end{aligned}
		\]
		then
		\[
		\left\|\mathbb{E}\left[p_{h}^{xy}\right]-E_h\left[p_{h}^{x y}\right]\right\|_{L^{1}(D)}
		\leqslant \frac{\sigma}{\sqrt{N}}\left\|p_h^{xy}  \right\|_{L^{1}(D \times \Omega)}~.
		\]
		Hence
		\[
		\left\|\mathbb{E}[p]-E_h\left[p_{h}^{xy}\right]\right\|_{L^{1}(D)}
		\leqslant C_1 \Delta x^{q}+C \Delta y^{r}+C_3\frac{\sigma}{\sqrt{N}}~.
		\]
		For the variance,
		\[
		\begin{aligned}
			\left\|\mathbb{V}[p]-V_{h}\left[p^{xy}_{h}\right]\right\|_{L^{1}(D)} &\leqslant \left\|\mathbb{E}\left[p^{2}\right]-E_h\left[\left(p_{h}^{x y}\right)^{2}\right]\right\|_{L^{1}(D)}+\left\|(\mathbb{E}[p])^{2}-\left(E_h\left[p_{h}^{x y}\right]\right)^{2}\right\|_{L^{1}(D)}, \\
			&\leqslant \left\|\mathbb{E}\left[p^{2}\right]-\mathbb{E}\left[\left(p_{h}^{x y}\right)^{2}\right]\right\|_{L^{1}(D)}+\left\|\mathbb{E}\left[\left(p_{h}^{x y}\right)^{2}\right]-E_h\left[\left(p_{h}^{x y}\right)^{2}\right]\right\|_{L^{1}(D)}\\
			&\ \ \ +\left\|(\mathbb{E}[p])^{2}-\left(E_h\left[p_{h}^{x y}\right]\right)^{2}\right\|_{L^{1}(D)}~.
		\end{aligned}
		\]
		The first term can be estimated by
		\[
		\begin{aligned}
			\left\|\mathbb{E}\left[p^{2}\right]-\mathbb{E}\left[\left(p_{h}^{x y}\right)^{2}\right]\right\|_{L^{1}(D)} 
			&=\int_{D}\left|\int_{\Omega}\left[p^{2}-\left(p_{h}^{x y}\right)^{2}\right] \rho(y) d y\right| d x\leqslant \int_{D} \int_{\Omega}\left|p^{2}-\left(p_{h}^{x y}\right)^{2}\right| \rho(y) d y d x\\
			& =\int_{D} \int_{\Omega}\left|p-p_{h}^{x y}\right|\left|p+p_{h}^{x y}\right| \rho(y) d y dx\leqslant C \int_{D} \int_{\Omega}\left|p-p_{h}^{x y}\right| \rho(y) d y d x\\
			&\leqslant C \int_{D} \int_{\Omega}\left|p-p_{h}^{y}\right| \rho(y) d y d x+C \int_{D} \int_{\Omega}\left|p_h^y-p_{h}^{xy}\right| \rho(y) d y d x\\
			&\leqslant C \sup_{\Omega} \rho(y)\int_{D} d x \int_{\Omega}\left|p-p_{h}^{y}\right| d y 
			+C\int_{\Omega} \rho(y) d y \int_{D} \left|p_h^y-p_{h}^{xy}\right|d x\\
			&=C \sup_{\Omega} \rho(y)\int_{D} \left\|p-p_{h}^{y}\right\|_{L^1(\Omega)}d x+C\int_{\Omega} \left\|p_h^y-p_{h}^{xy}\right\|_{L^1(D)}\rho(y) d y\\
			&\leqslant C\sup_{\Omega}\rho(y)|D|C_2\Delta y^{r}+CC_1\Delta x^{q}~.
		\end{aligned}
		\]
		For the second integral, we have
		\[
		\begin{aligned}
			\left|\mathbb{E}\left[(p_h^{xy})^{2}\right]_i^n-E_{h}\left[(p_{h}^{xy})^{2}\right]_{i}^{n}\right|&=\left| \int_{\Omega} \sum_{k}\left(p_{i k}^{n}\right)^{2} \rho(y) d y-\sum_{k}\left(p_{i k}^{n}\right)^{2} {\tilde\alpha_k} \right| \\
			&\leqslant \sum_{k}\left(p_{i k}^{n}\right)^{2} \left|\left(\int_{\omega_{k}} \rho(y) d y -{\tilde\alpha_k}\right) \right| \\
			&\leqslant \frac{\sigma}{\sqrt{N}}\sum_{k}\left(p_{i k}^{n}\right)^{2}~,
		\end{aligned}
		\]
		hence
		\[
		\left\|\mathbb{E}\left[\left(p_{h}^{x y}\right)^{2}\right]-E_h\left[\left(p_{h}^{x y}\right)^{2}\right]\right\|_{L^{1}(D)} \leqslant  \left\|(p_h^{xy})^2\right\|_{L^{1}(D \times \Omega)}\frac{\sigma}{\sqrt{N}}\leqslant C\frac{\sigma}{\sqrt{N}}.
		\]
		For the third integral, we get
		\[
		\begin{aligned}
			\left\|(\mathbb{E}[p])^{2}-\left(E_h\left[p_{h}^{x y}\right]\right)^{2}\right\|_{L^{1}(D)} &=\int_{D}\left|(\mathbb{E}[p])^{2}-\left(E_h\left[p_{h}^{x y}\right]\right)^{2}\right| d x\\
			&=\int_{D}\left|\mathbb{E}[p]-E_h\left[p_{h}^{x y}\right]\right|\left|\mathbb{E}[p]+E_h\left[p_{h}^{x y}\right]\right| d x,\\
			&\leqslant C\left\|\mathbb{E}[p]-E_h\left[p_{h}^{x y}\right]\right\|_{L^{1}(D)},
		\end{aligned}~.
		\]
		Finally,
		\[
		\left\|\mathbb{V}[p]-V_{h}\left[p^{xy}_{h}\right]\right\|_{L^{1}(D)} \leqslant C_1^{'} \Delta x^{q}+C^{'} \Delta y^{r}+C_3^{'}\frac{\sigma}{\sqrt{N}}~.
		\]
	\end{proof}

	\section{Test Cases}
	\subsection{The Kraichnan-Orszag three-mode problem with 1D Random Parameter}
	In this test case, we consider the Kraichnan-Orszag three-mode problem governed by Eq.~\eqref{K-O} where $y$ is a uniformly distributed random variable in [-1,1]. The initial conditions for the dependent variables $z_1, z_2$ and $z_3$ are
	\begin{equation}
		z_1(0) = 1.0,\quad z_2(0) = 0.1y,\quad z_3(0) = 0.0,
		\label{eq:initial_conditions}
	\end{equation}
	
	The exact solution is given as in \citet{wan2005}. The quasi-Monte Carlo (QMC) method using SOBOL sequence is adopted to generate samples of $y$. For SFV-Cartesian, the interval [-1, 1] is divided into a Cartesian grid corresponding to the desired number of clusters. For SFV, the cells in the parameter space are built by clustering the QMC samples using K-means. The solutions on grid cells or clusters are used to calculate the expectation and variance of solution.
	
	The QMC solution obtained from $2^{17}$ samples is employed to calculate the reference expectation and variance for validation. It is observed in Fig.~\ref{experiment_clusters} that both the errors of expectation and variance gradually decrease over number of clusters for fixed number of samples. The convergence rate of the expectation and variance are found to be close to each other via data fitting, corresponding to the \(r\) in Theorem 4.2. This is in accordance with Theorem 4.2. 
	
	\begin{figure}[!htp]
		\centering
		\includegraphics[width=0.9\linewidth]{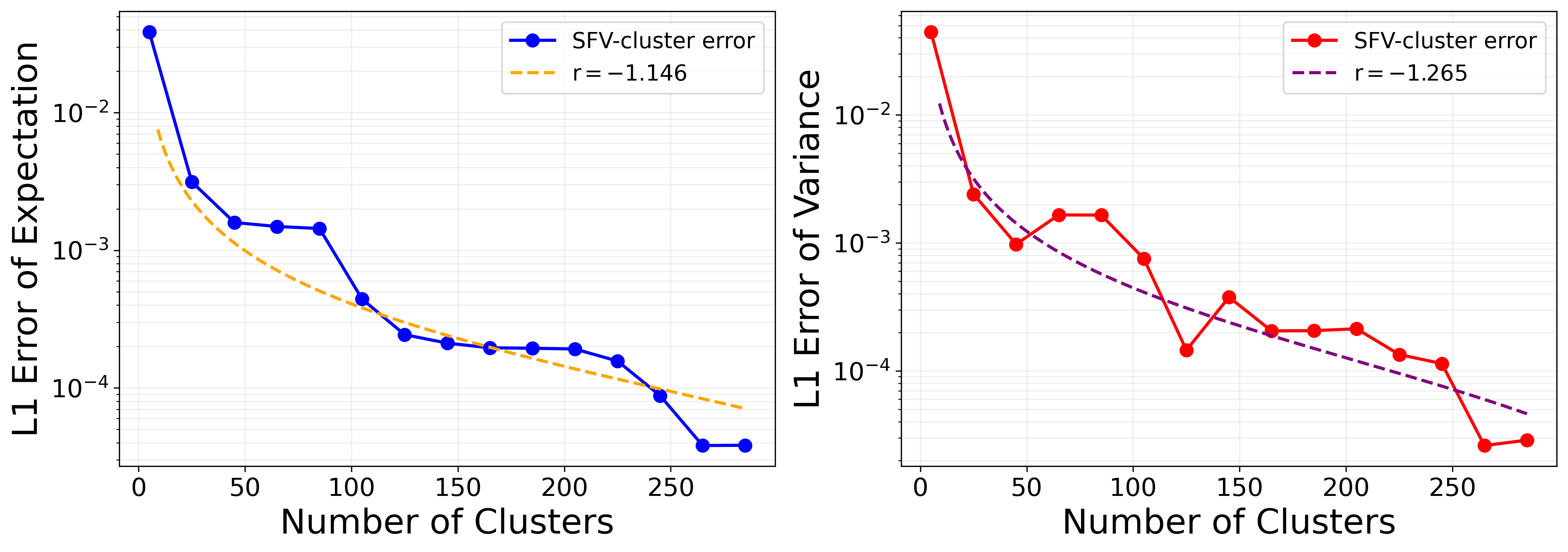}
		\caption{The convergence of the SFV-cluster method for the Kraichnan–Orszag three-mode problem over  number of clusters: expectation (left) and variance (right). The convergence rate of the expectation and variance are close to each other which is in accordance with Theorem 4.2.}
		\label{experiment_clusters}
	\end{figure}
	
	The convergence of the SFV and the SFV-cluster are compared in Fig.~\ref{1experiment_compare_sfv_structured}. The convergence rate are approximately identical for this simple 1D problem. 
	A comparison between the SFV-cluster method and QMC is shown in Fig.~\ref{experiment_compare_sfv_vs_mc}. The errors of expectation and variance obtained by SFV-cluster are significantly lower than those of QMC given equivalent number of forward simulation runs.  In addition, the convergence rate of SFV-cluster is higher than that of QMC. 
	\begin{figure}[!htp]
		\centering
		\includegraphics[width=0.9\linewidth]{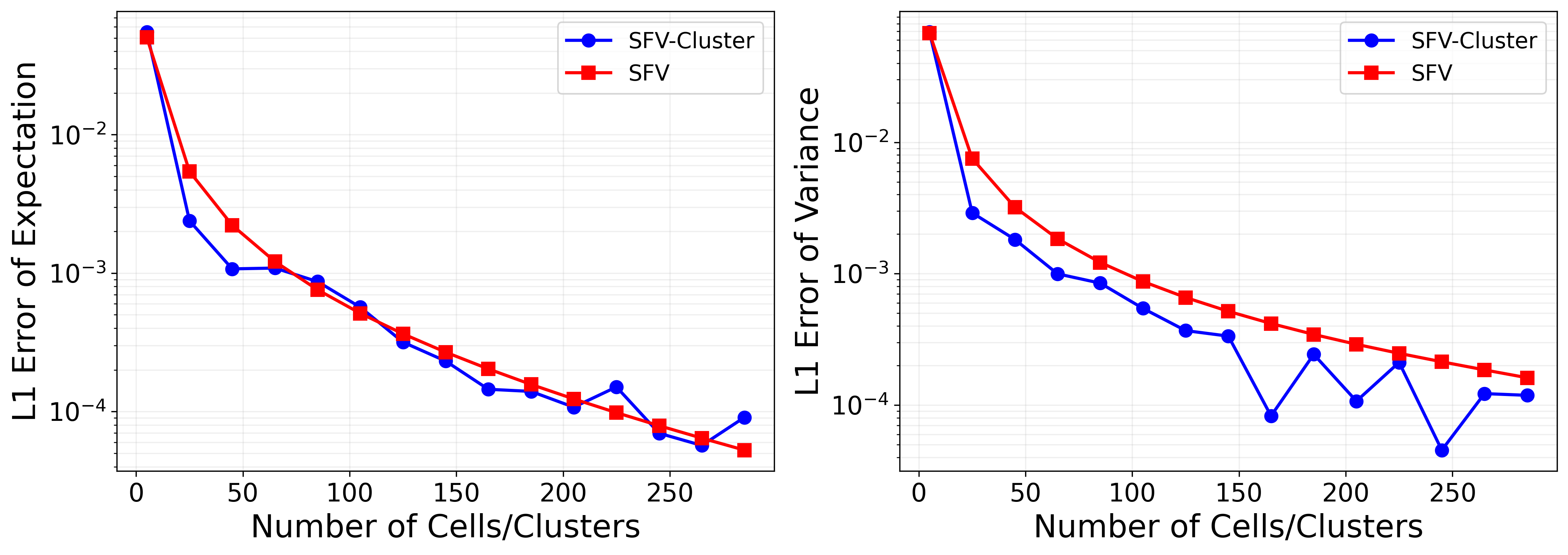}
		\caption{The convergence of the mean (left) and variance (right) of L1 error computed by SFV and SFV-cluster.}
		\label{1experiment_compare_sfv_structured}
	\end{figure}
	\begin{figure}[!htp]
		\centering
		\includegraphics[width=0.9\linewidth]{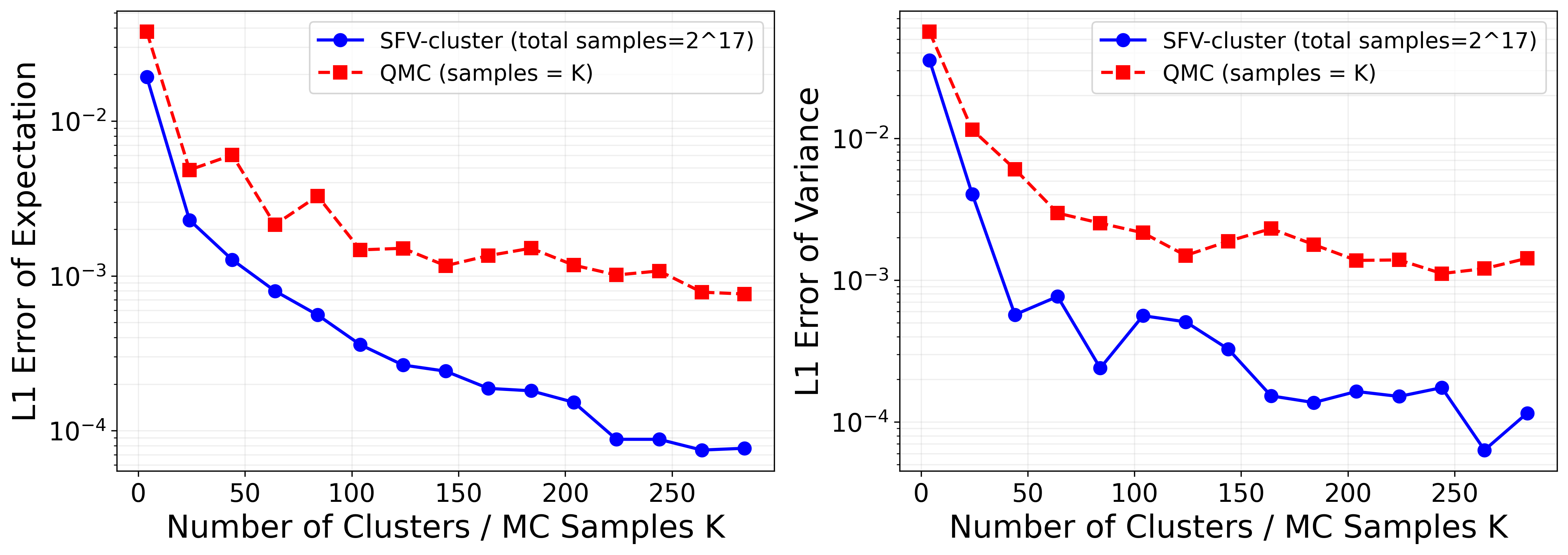}
		\caption{The error of expectation computed by SFV-cluster and QMC (left) and the corresponding error of variance (right). This implies that the result of SFV-cluster is more accurate than QMC given identical cost of forward simulation runs.}
		\label{experiment_compare_sfv_vs_mc}
	\end{figure}
	
	\subsection{The Kraichnan-Orszag three-mode problem with 2D Random Parameter}
	In this test case, we consider the Kraichnan-Orszag three-mode problem governed by Eq.~\eqref{K-O} where \(y_1\) is discontinuous on \([-1,1]\) with probability density function defined by Eq.~\eqref{pdf} with the constant \(M=11\) for normalization, and \(y_2\) is a uniform random variable in \([-1,1]\).
	
	\begin{equation}
		f(y_1) = \frac{1}{M} \times
		\begin{cases}
			\dfrac{1 + \cos(\pi y_1)}{2} & \text{if } y_1 \in [-1, 0], \\
			10 + \dfrac{1 + \cos(\pi y_1)}{2} & \text{if } y_1 \in [0, 1], \\
			0 & \text{else}
			\label{pdf}
		\end{cases}
	\end{equation}
	
	\begin{equation}
		z_1(0) = 1.0,\quad z_2(0) = 0.1y_1,\quad z_3(0) = y_2,
		\label{eq:initial_conditions}
	\end{equation}
	
	The QMC solution obtained from $2^{17}$ samples is employed to calculate the reference expectation and variance for validation. The convergence of errors for the expectation and variance approximated by SFV-cluster is shown in Fig.~\ref{2experiment_clusters_english}. The approximated convergence rates for the expectation and variance are close to each other, which is in accordance with Theorem 4.2. Using five randomly chosen clusters, the projected solution trajectories are close to the reference \citep{wan2005} (see Fig.~\ref{2clusters_100_multiple_cluster_3d_trajectories} and Table.~\ref{tab:cluster_errors}). 
	
	\begin{figure}[!htp]
		\centering
		\includegraphics[width=0.9\linewidth]{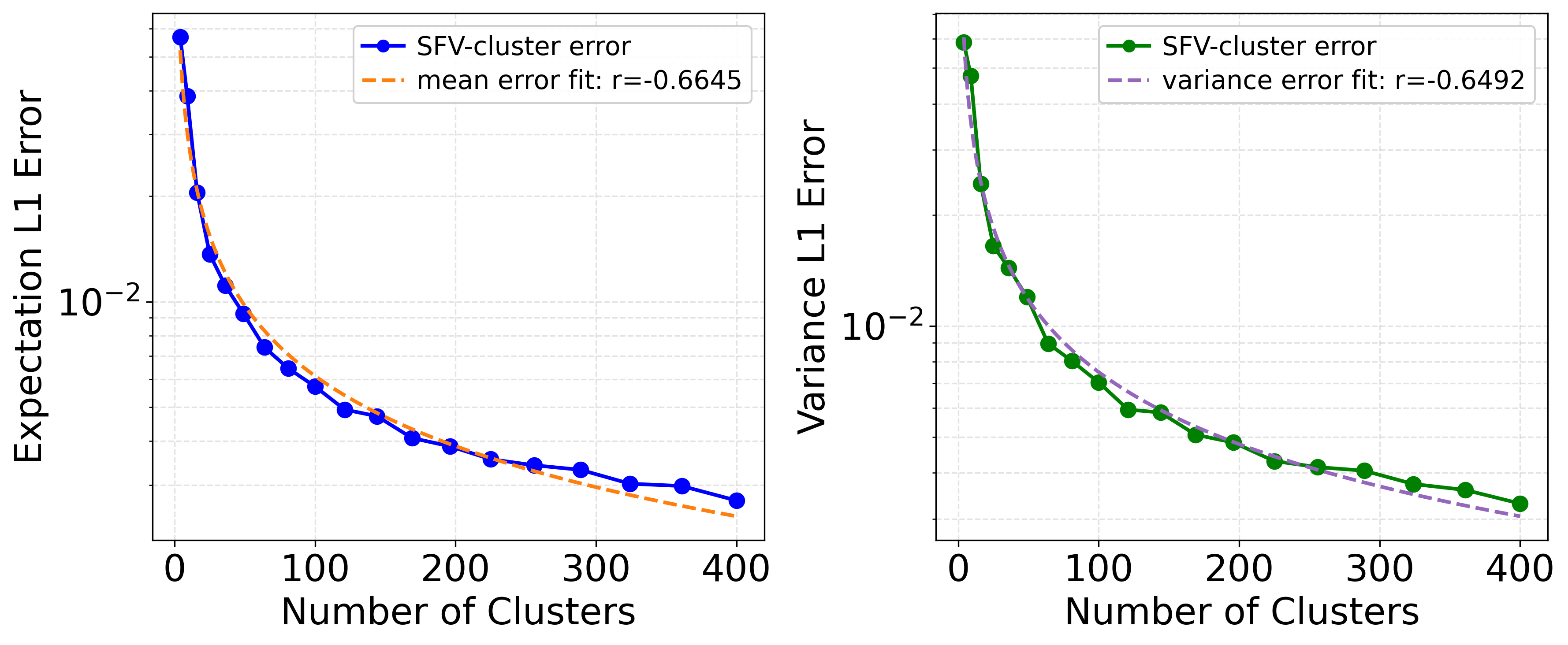}
		\caption{The variation of the error of the SFV-cluster method for the Kraichnan-Orszag three-mode problem with 2D random parameters as a function of the number of clusters: the mean of the error (left) and the variance of the error (right)}
		\label{2experiment_clusters_english}
	\end{figure}

	\begin{figure}[!htp]
		\centering
		\includegraphics[width=0.99\linewidth]{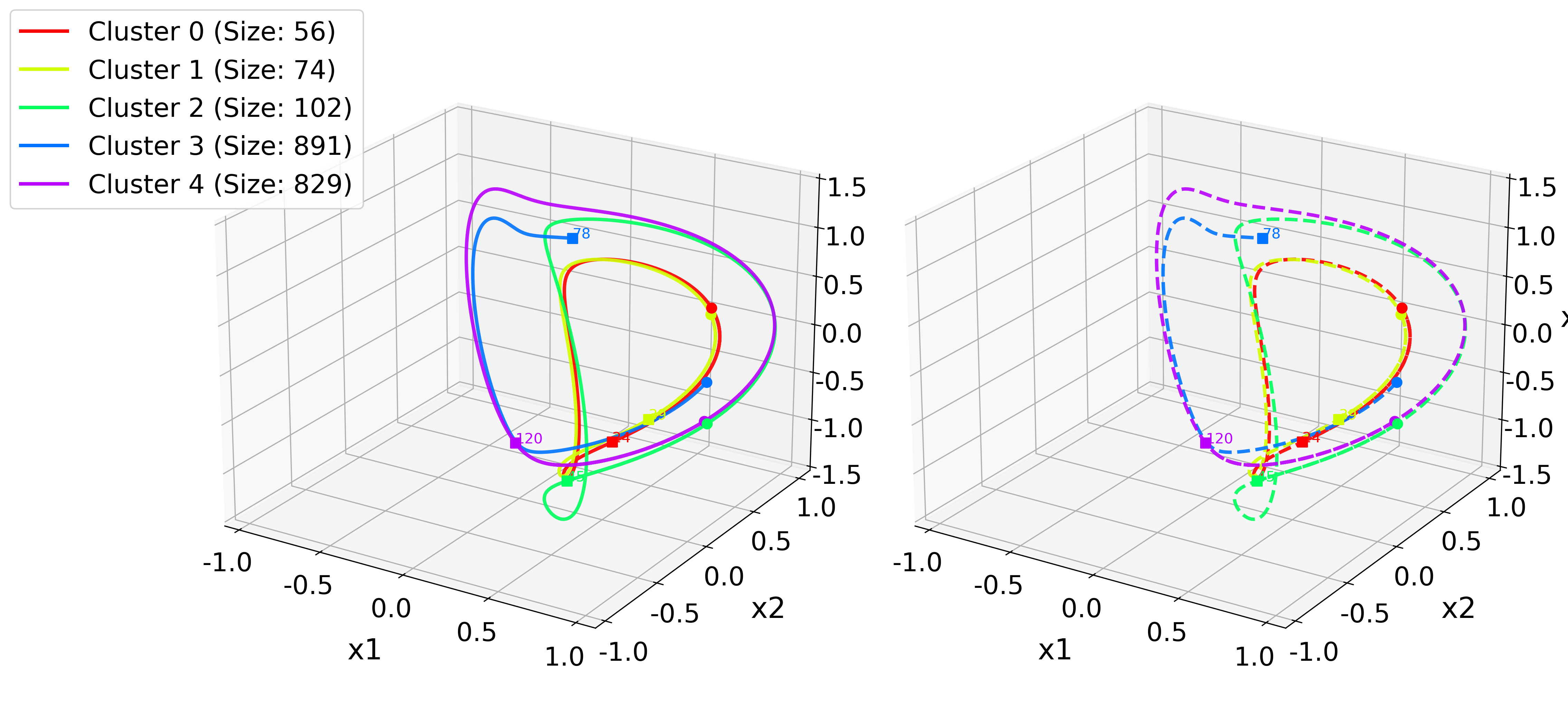}
		\caption{The solutions of the Kraichnan–Orszag problem with 2D random parameters obtained by the SFV-cluster method under different initial conditions. Numerical solution  (left); Analytic solution  (right)}
		\label{2clusters_100_multiple_cluster_3d_trajectories}
	\end{figure} 
	\begin{table}[htbp]
		\centering
		\caption{Errors between the SFV-cluster and analytic solutions for each cluster chosen randomly.}
		\label{tab:cluster_errors}
		\begin{tabular}{ccc}
			\toprule
			Cluster index & Mean absolute error ($\bar{e}$) & Max norm error ($e_{\infty}$) \\
			\midrule
			0 & $3.99\times10^{-7}$ & $1.94\times10^{-6}$ \\
			1 & $2.77\times10^{-7}$ & $1.83\times10^{-6}$ \\
			2 & $4.17\times10^{-7}$ & $3.31\times10^{-6}$ \\
			3 & $1.13\times10^{-7}$ & $7.37\times10^{-7}$ \\
			4 & $1.25\times10^{-7}$ & $8.11\times10^{-7}$ \\
			\bottomrule
		\end{tabular}
	\end{table}
	\begin{figure}[!htp]
		\centering
		\includegraphics[width=0.99\linewidth]{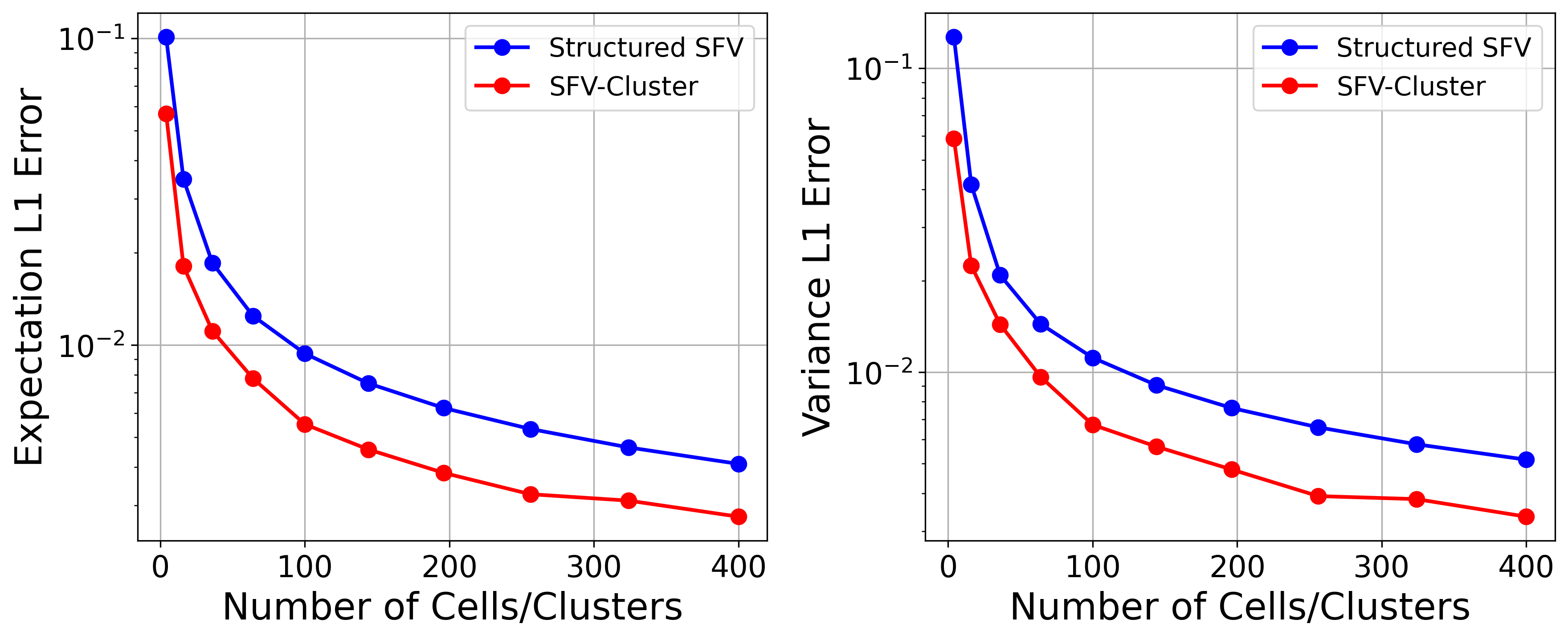}
		\caption{The mean of error computed by SFV and SFV-cluster (left) and the variance of error computed by SFV and SFV-cluster (right).}
		\label{2method_comparison_english}
	\end{figure} 
	Under a fixed sample size, the variations of errors for the SFV and SFV-cluster methods over the number of cells are illustrated in Fig.~\ref{2method_comparison_english}. The convergence of SFV-cluster is faster than that of SFV using Cartesian grid. The trajectories of expectation and variance over time obtained by QMC, SFV, and SFV-cluster method are compared in Fig.~\ref{2methods_comparison_units_100}, validating the accuracy of SFV-cluster. 
	\begin{figure}[!htp]
		\centering
		\includegraphics[width=0.99\linewidth]{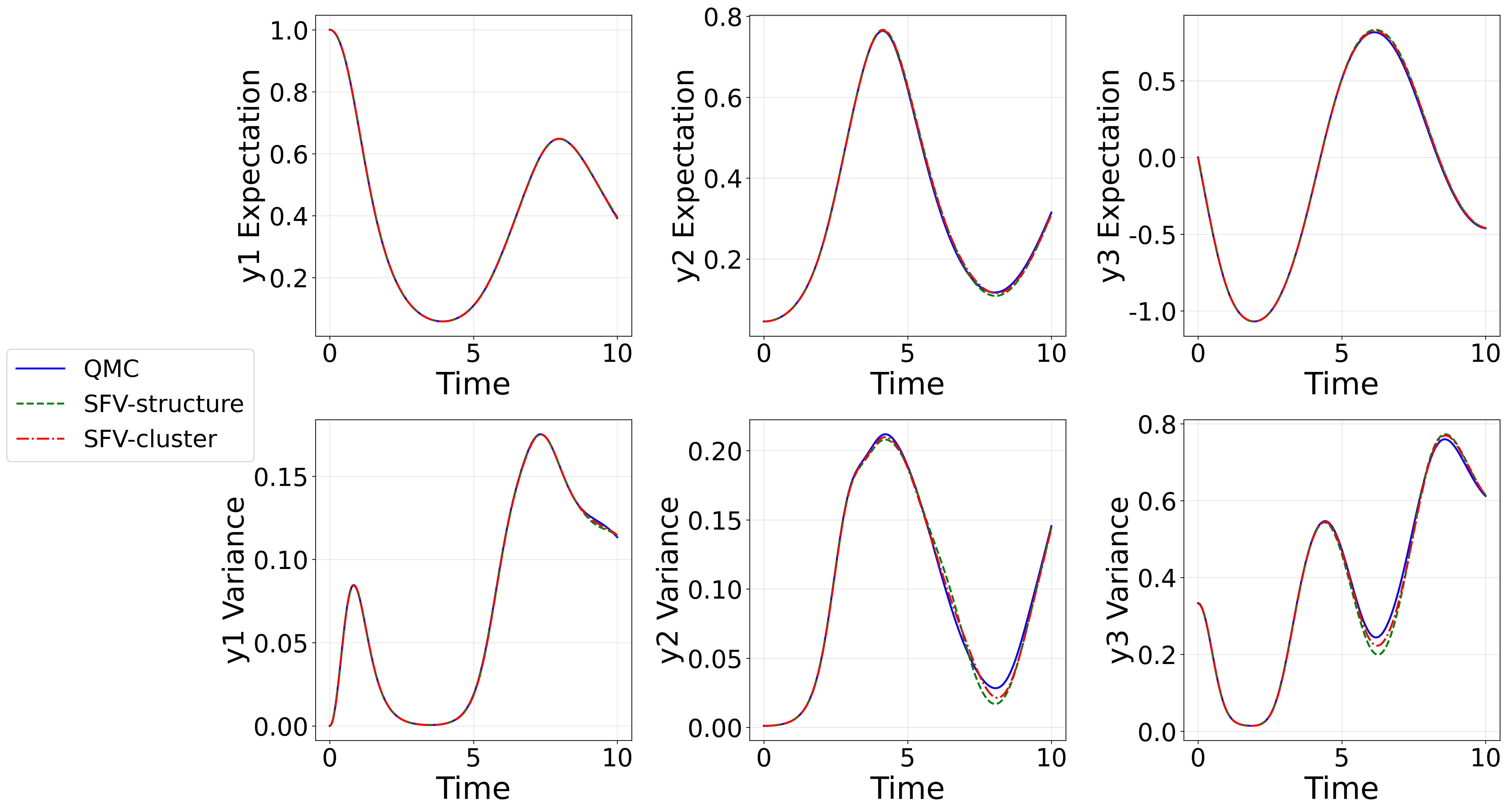}
		\caption{Comparison of expectation and variance for \( y_1, y_2, y_3 \) estimated by QMC, SFV-structure, and SFV-cluster for the Kraichnan–Orszag three‑mode problem with 2D random parameters.}
		\label{2methods_comparison_units_100}
	\end{figure} 
	
	\subsection{Buckley-Leverett Equation with 5D Random Parameter}
	In this test case, we consider the Buckley-Leverett example governed by Eq.~\eqref{buckley_leverett} with 5D random parameter. The physical domain $x \in [0,6]$ is uniformly partitioned into five segments. The random variable $\alpha(\omega)$ is defined piecewise over these segments. On each segment, $\alpha(\omega)$ follows a uniform distribution within each interval of $[0,0.15]$, $[0.25,0.55]$, $[0.65,0.75]$, $[0.85,1.05]$, and $[1.15,1.55]$, respectively. The five dimensions of $\alpha(\omega)$ are independent.

	
	For the physical domain $[0,6]$, there are 200 grid cells. 
	QMC with $2^{17}$ samples is employed to obtain the reference expectation and variance. The solutions of SFV, SFV-cluster and QMC at time $t=1$ are compared. The convergence of expectation and variance for SFV-cluster over number of clusters is shown in Fig.~\ref{experiment_vary_clusters_simplified}. The spatial profiles of the expectation with respect to x computed by QMC and SFV-cluster are shown in Fig.~\ref{final_time_comparison_t_-1}. The nonlinear flux $f(u;y)$ and the spatially varying random parameter $\alpha(y)$ lead to different local wave speeds, producing sharp fronts in the solution profile. The close agreement between the QMC and SFV-cluster profiles in these rapidly varying regions indicates that the SFV-cluster method accurately reproduces the numerical solution of the Buckley-Leverett equation. The SFV method based on structured grid is not included in comparison as generating structured grid in a 5D parameter space leads to very high number of cells that is computationally expensive for SFV.
	
	\begin{figure}[!htp]
		\centering
		\includegraphics[width=0.99\linewidth]{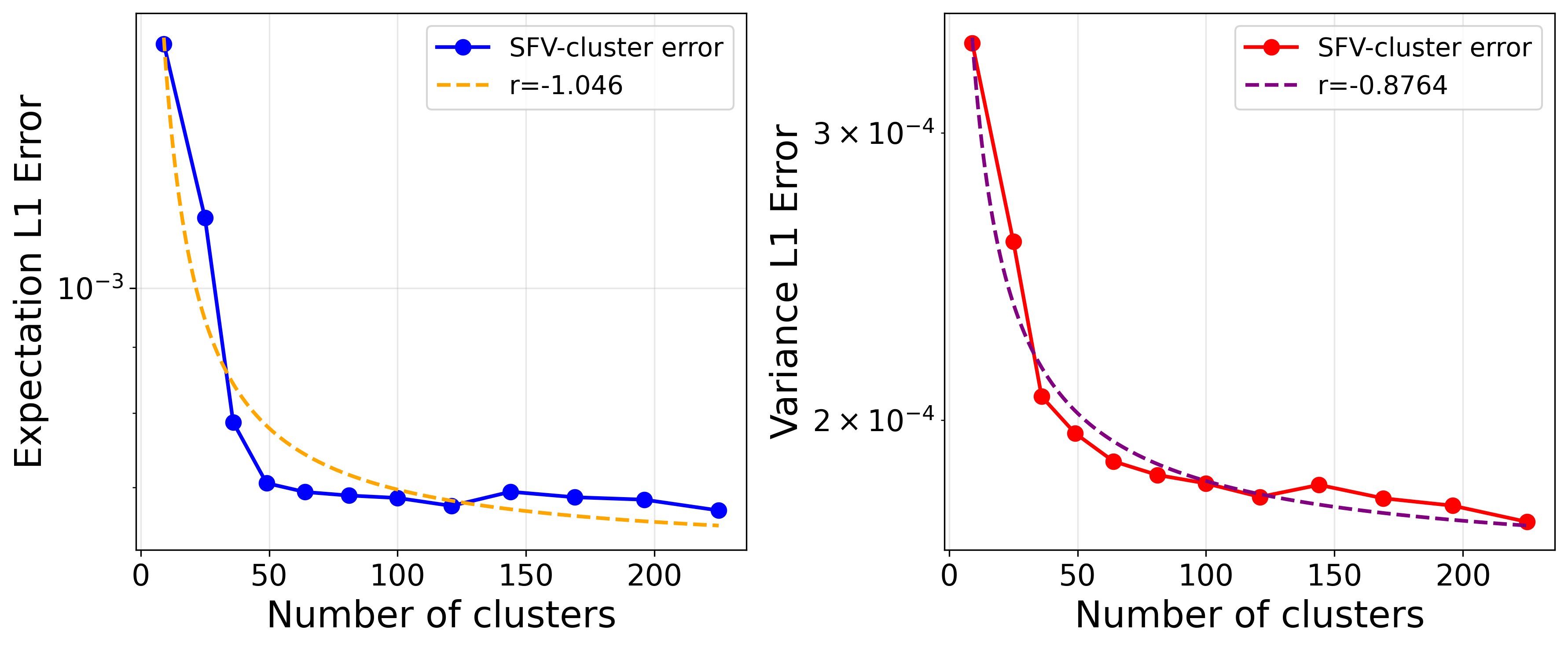}
		\caption{The convergence of the expectation (left) and variance (right) for SFV-cluster in the Buckley-Leverett test case over the number of clusters.}
		\label{experiment_vary_clusters_simplified}
	\end{figure}
	
	\begin{figure}[!htp]
		\centering
		\includegraphics[width=0.75\linewidth]{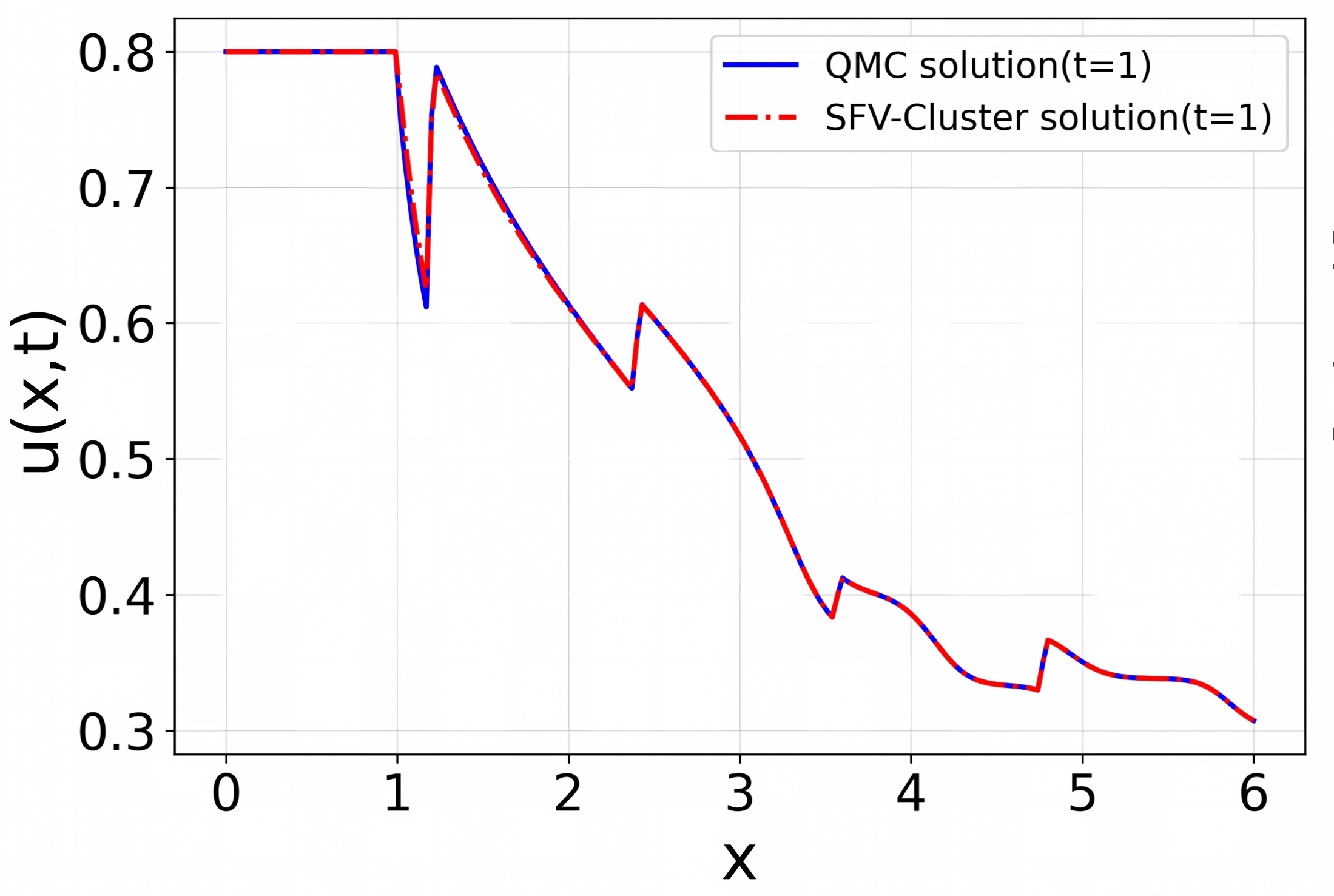}
		\caption{Comparing solutions at $t=1$ by the QMC method and the SFV-cluster method for the Buckley–Leverett equation with a 5D random parameter.}
		\label{final_time_comparison_t_-1}
	\end{figure}

	\newpage
	\section{Conclusions}
	In the current study,  we propose a new SFV-cluster scheme for forward UQ where samples in the parameter space are clustered to form cells with implicit boundaries. Since there is no flux between cells in the parameter space assuming first-order SFV, the cells are decoupled entirely. Therefore, unstructured mesh with implicit boundaries can be built in the parameter space by clustering samples. Employing the new SFV scheme, we only need to conduct forward simulation on the clusters, thus reducing the number of simulation runs for estimating the mean and variance compared to QMC. The new scheme has been validated on ODE and hyperbolic PDE problems with random parameters. Although K-means can be used to for clustering in parameter spaces with dimension higher than 3, and there is no explicit limit on the number of dimensions, caution is needed for applying K-means to much higher dimensions, as Euclidean distance gradually becomes ineffective as the number of dimensions grows higher. In such cases, more effective clustering or dimensional reduction methods could be employed.
	\section*{Data Availability}
	The data that support the findings of this study are available from the corresponding author upon reasonable request.
	
	\section*{Acknowledgements}
	The research is supported by the Natural Science Foundation of Shandong Province (No. ZR2024MA057), the Natural Science Foundation of Jiangsu Province (No. BK20220272), Qingdao Science and Technology Bureau (23-1-2-qljh-3-gx), the Fundamental Research Funds for the Central Universities and the Future Plan for Young Scholars of Shandong University. 
	\newpage

	
	
	\bibliographystyle{elsarticle-harv} 
	\bibliography{mybib}
	
	
	
	\newpage

\end{document}